\documentclass[10pt]{article}

\usepackage{amssymb}
\usepackage{amsfonts}
\usepackage{amsmath}
\usepackage{amssymb}
\usepackage{graphicx}
\usepackage{caption}
\usepackage{float}
\usepackage[utf8]{inputenc}
\usepackage[english]{babel}
\usepackage{amsthm}

\newtheorem{theorem}{Theorem}
\newtheorem{lemma}[theorem]{Lemma}

\begin{document}
			
		\title{Multicoloring of Graphs to Secure a Secret}
		

		\author{Tanja Vojkovi\'c (1) \footnote{Corresponding author: tanja@pmfst.hr}, Damir Vuki\v{c}evi\'c (1), Vinko Zlati\'c (2,3)}

			
		\maketitle
		
		\noindent (1) Department of Mathematics, Faculty of Science, Split, Croatia\\
		(2) Ru\dj er Bo\v{s}kovi\'c Institute, Zagreb, Croatia\\
		(3) CNR Institute for Complex Systems, via dei Taurini 19, 00185, Rome, Italy\\

		\begin{abstract}
			Vertex coloring and multicoloring of graphs are a well known subject in graph theory, as well as their applications. In vertex multicoloring, each vertex is assigned some subset of a given set of colors. Here we propose a new kind of vertex multicoloring, motivated by the situation of sharing a secret and securing it from the actions of some number of attackers. We name the multicoloring a highly $a$-resistant vertex $k$-multicoloring, where $a$ is the number of the attackers, and $k$ the number of colors. For small values $a$ we determine what is the minimal number of vertices a graph must have in order to allow such a coloring, and what is the minimal number of colors needed.
		\end{abstract}
	
		
		05C82, 05C15, 68R10, 94A62 
		%%MSC codes here, in the form: \MSC code \sep code

\section{Introduction}
\label{}

Vertex multicolorings are generalizations of ordinary graph colorings in which every vertex is assigned a set of colors instead of one color. A proper multicoloring, just as the proper coloring, means that adjacent vertices cannot have the same color, so the sets of colors of adjacent vertices are disjoint. As in the ordinary colorings, the main problem in multicoloring is to minimize the number of colors used, but other problems and objectives have also been explored \cite{M1,M2,M3,M4}.
Applications of coloring and multicoloring of graphs are many and well known, from map colorings and Sudoku, to scheduling and frequency allocation problems.\\
In this paper we define a new type of vertex multicoloring, motivated by a problem of securing a secret. The idea is to safeguard a secret (a message or an information) by dividing it into parts and distributing those parts amongst the participants of some group. The model assumes a number of participants is corrupted and the number of parts and the distribution is determined in a way so that some subset of the participants can reconstruct the secret after the attack of the corrupted parties.
In our paper, these parts of the secret are modeled by colors so determinig the minimal number of parts and their distribution is the problem of finding a minimal number of colors and the exact coloring function.\\
Let $a\in\mathbb{N}$. We will say that a coloring is $a$-resistant vertex multicoloring if the following holds:
If we remove any $a$ vertices, and their neighbors, from the graph, in the remaining graph there exists a component in which all the colors are present.
$a$-resistant vertex multicoloring is called highly $a$-resistant vertex multicoloring  if any $a$ vertices do not have all the colors.\\
These types of coloring are motivated by the following situation:
Some organization has planted a group of sleepers (spies that live normal life until they are called to perform some mission). Their mission is a secret, divided into parts and distributed among them. Each sleeper can have some or none of the parts. Further, each of the sleepers knows only some of his colleagues. These sleepers can be represented as a graph in which edges connect pairs of sleepers that know each other. The mission can be implemented if there is a connected component of this graph that has all the parts. The assumption is that there are $a$ adversary's agents planted in the sleepers group and if there is an adversary agent among the sleepers, he will betray all the sleepers he knows and give his parts of the secret to the adversary.\\
So each part is represented by a color and if a graph admits a highly $a$-resistant vertex multicoloring then the conditions are fulfilled and this graph is resistant to $a$ adversary agents behaving in the described way.\\
We will first formulate the problem mathematically and then analyze the minimal number of colors for highly $a$-resistant vertex multicoloring for some fixed $a$, namely for $1,2,3$ and $4$. We determine what is the minimal number of vertices a graph must have in order to have a highly $a$-resistant vertex multicoloring, for a fixed $a$, and what is the minimal number of colors needed for coloring of such a graph. For each case we propose an example of a highly $a$-resistant vertex multicolored graph.
We do not ask that highly $a$-resistant vertex multicoloring is proper. However we will see that the given conditions imply that the coloring with minimal number of colors will indeed be proper.

\section{Formulating the problem}

Let us formulate this problem in a mathematical way.\\

Standardly, as in \cite{bollobas}, we denote $G=(V,E)$, where $G$ is a graph with the set of vertices
$V=V(G)$, $|V(G)|=v(G)$, and the set of edges $E=E(G)$. Let $u\in V$ and
$A\subseteq V$. We denote:\\

$N(u)=N_{G}(u)$ the set of neighbors of $u$ in $G$;\\

$M(u)=M_{G}(u)=N(u)\cup\{u\}$;\\

$N(A)=N_{G}(A)={\displaystyle\bigcup\limits_{u\in A}}N_{G}(u)$;\\

$M(A)=M_{G}(A)={\displaystyle\bigcup\limits_{u\in A}}M_{G}(u)$.\\

Further, with $G\backslash A$ we denote a graph obtained from $G$ by deletion
of the vertices in $A$ and their incident edges. If $H_{1}$ and $H_{2}$ are graphs, with $G=H_{1}\cup H_{2}$ we denote a graph where $V(G)=V(H_{1})\cup V(H_{2})$ and $E(G)=E(H_{1})\cup E(H_{2})$. For any set $T$, with
$\mathcal{P}(T)$ we denote the partitive set of $T$.\\

\textbf{Vertex $k$-multicoloring} of graph $G$ is a function $\kappa:V(G)\rightarrow\mathcal{P}(\{1,...,k\})$, where each vertex is colored with some subset of the set of $k$ colors.

Let $a\in\mathbb{N}$. Vertex $k$-multicoloring of a graph $G$ is called \textbf{$a$-resistant vertex $k$-multicoloring} if the following holds:

For each $A\subseteq V(G)$ with $a$ vertices, there is a component $H$ of the graph $G\backslash M_{G}(A)$ such that
$${\displaystyle\bigcup\limits_{u\in V(H)}}\kappa(u)=\{1,...,k\}\text{.}$$

$a$-resistant vertex $k$-multicoloring is called \textbf{highly $a$-resistant vertex $k$-multicoloring} if for each $A\subseteq V(G)$ with $a$ vertices it holds that

\begin{center}
	${\displaystyle\bigcup\limits_{u\in A}}\kappa(u)\neq\{1,...,k\}$.
\end{center}

In other words, vertex $k$-multicoloring is $a$-resistant, for some $a\in \mathbb{N}$, if for each subset $A\subseteq V(G)$ with $a$ vertices, there exists a component $H$ of the graph $G\backslash M_{G}(A)$ such that all $k$ colors are present in $H$. In this case we will sometimes say that $H$ \textit{has all $k$ colors}.
If, in addition to that it holds that none $A\subseteq V(G)$ with $a$ vertices has all the colors, we will say that the multicoloring is highly $a$-resistant.\\
We will denote by $a-HR$ the condition that 
	${\displaystyle\bigcup\limits_{u\in A}}\kappa(u)\neq\{1,...,k\}$, for each $A\subseteq V(G)$ with $a$ vertices. So the multicoloring is highly $a$-resistant if it is $a$-resistant and the $a-HR$ condition holds.\\

It is easily seen that the following holds:\\
i) For some graph $G$ and the multicoloring of its vertices the $a-HR$ condition may hold for some $a\in\mathbb{N}$ but the multicoloring may not be $a$-resistant.\\
ii) If a vertex $k$-multicoloring is not $a$-resistant, for some $a\in \mathbb{N}$, than it is not $b$-resistant, for each $b\in \mathbb{N}, a\leq b$.\\
iii) If $k\leq a$ the $a-HR$ condition cannot hold for any graph $G$ and any $k$-multicoloring of its vertices.\\
iv) If for some $a\in\mathbb{N}$ the $a-HR$ condition holds for $\kappa$ in some graph $G$, then it also holds for any subgraph of $G$.

\section{Results}

\begin{theorem}\label{tm1}
	Let $G$ be a graph with $n$ vertices, $a,k\in\mathbb{N}$, and $\kappa$ a highly $a$-resistant vertex $k$-multicoloring of $G$. It holds:
	
	1) There exists a graph with $n+1$ vertices that admits a highly $a$-resistant vertex $k$-multicoloring.
	
	2) There exists a highly $a$-resistant vertex $k+1$-multicoloring $\kappa'$ of $G$. 
\end{theorem}
\begin{proof}
	Let $G$ be a graph with $n$ vertices, $a,k\in\mathbb{N}$, and $\kappa$ a highly $a$-resistant vertex $k$-multicoloring of $G$.
	
	1) Let $G'$ be defined as $G^{\prime}=(V^{\prime
	},E^{\prime})$, $V^{\prime}=V(G)\cup\{u\}$, $E^{\prime}=E(G)$, and $\kappa'$ a vertex $k$-multicoloring of $G'$ such that $\kappa^{\prime}|_{V^{\prime}\backslash\{u\}}=\kappa$, $\kappa^{\prime}(u)=\emptyset$. It is easy to see that $\kappa'$ is higly $a$-resistant vertex $k$-multicoloring of $G'$.

	2) Let $\kappa'$ be a vertex $(k+1)$-multicoloring of $G$ such that $\kappa^{\prime}(u)=\kappa(u)\cup\{k+1\}$, for
	each $u\in V(G)$. It is easy to see that $\kappa'$ is highly $a$-resistant vertex $(k+1)$-multicoloring.
\end{proof}

Theorem \ref{tm1} implies that if we determine that a highly $a$-resistant vertex $k$-multicoloring doesn't exist for any graph with $n$ vertices, for given $n,a,k\in\mathbb{N}$, then such coloring doesn't exist for any graph with less then $n$ vertices. Also, if for some $n,a,k\in\mathbb{N}$ there exists a graph with $n$ vertices that admits a highly $a$-resistant vertex $k$-multicoloring, then we can multicolor vertices of $G$ with more then $k$ colors and it will also be highly $a$-resistant.
This compels us to search for a minimal number $k$ of colors, for which, for given $a$, there exist $n\in\mathbb{N}$ such that a graph with $n$ vertices admits a highly $a$-resistant vertex $k$-multicoloring, and $n$ is the smallest such number.

Let us denote by $K(a,n)$ a minimal number of colors such that there exists a graph $G$ with $n$ vertices and a highly $a$-resistant vertex multicoloring of $G$ with $K(a,n)$ colors.

If, for some $a$ and $n$ a graph with $n$ vertices that admits a highly $a$-resistant vertex $k$-multicoloring doesn't exist for any $k\in\mathbb{N}$ we will say that $K(a,n)=\infty$.

\begin{theorem}
	It holds
	\[K(1,n)=\left\{
	\begin{array}
	[c]{cc}
	\infty,&  n\leq3\text{;}\\
	2, &  n\geq4\text{.}
	\end{array}
	\right.
	\]
\end{theorem}

\begin{proof}
	First, let us prove that a graph with $3$ vertices doesn't admit a highly $1$-resistant vertex $k$-multicoloring, for any $k\in\mathbb{N}$. Suppose to the contrary that there exists $k\in\mathbb{N}$, a graph $G$ with $3$ vertices and a multicoloring of vertices of $G$ with $k$ colors that is highly $1$-resistant.
	Let us denote $V(G)=\{v_{1},v_{2},v_{3}\}$. It is easy to see that the claim holds if $E(G)$ is empty, so let us assume $E(G)$ is non-empty. Without the loss of generality let us assume $v_{1}v_{2}\in E(G)$. If $A=\{v_{1}\}$ then
	in order for $\kappa$ to be $1$-resistant it must hold $\kappa(v_{3})=\{1,..,k\}$, but than it is obviously not highly $1$-resistant.A graph with $4$ vertices with a highly $1$-resistant vertex $2$-multicoloring is given in Figure \ref{fig:f1}.
	
	\begin{figure}[h]				
		\centering
		\includegraphics[scale=0.7]{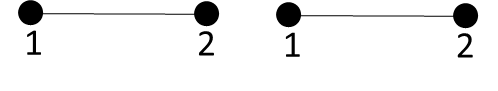}
		\caption{A graph with $4$ vertices with a highly $1$-resistant vertex $2$-multicoloring}
		\label{fig:f1}\ \
	\end{figure}

	Now the claim follows from Theorem \ref{tm1}.
\end{proof}

\begin{theorem}
	It holds
	\[
	K(2,n)=\left\{
	\begin{array}
	[c]{cc}%
	\infty, & n\leq8\text{;}\\
	3, & n\geq9\text{.}
	\end{array}
	\right.
	\]
\end{theorem}

\begin{proof}
	First, let us prove that a graph with $8$ vertices doesn't admit a highly $2$-resistant vertex $k$-multicoloring, for any $k\in\mathbb{N}$. Suppose to the contrary, that there exists $k\in\mathbb{N}$, a graph $G$ with $8$ vertices and a multicoloring of vertices of $G$ with $k$ colors that is highly $2$-resistant. Note that in order for $2-HR$ condition to hold, for each ${u,v}\subseteq V(G)$, $\kappa(u)\cup \kappa(v) \neq \{1,...,k\}$, that is, no two vertices have all the colors. Hence, there is a component with at least three vertices in $G$. Therefore, there is a vertex of degree at least
	two, let us denote it by $v_{1}$. Let us observe the graph $G_{1}=G\backslash
	M(v_{1})$ that has at most $5$ vertices. If $G_{1}$ has no component with more
	than $2$ vertices, then the multicoloring is not $2$-resistant because components with at most $2$ vertices don't have all the colors. Otherwise, there is a vertex of degree at least $2$ in $G_{1}$, let us
	denote it by $v_{2}$. Then the graph $G_{2}=G_{1}\backslash M(v_{2})$ has at
	most two vertices, and again, the multicoloring cannot be $2$-resistant and highly $2$-resistant. Hence, indeed $K(2,n)=\infty$. The fact that
	$K(2,n)=3$ for $n\geq9$ follows from Theorem \ref{tm1} and Figure \ref{fig:f2}.
	
	\begin{figure}[h]
		\centering\includegraphics[scale=0.3]{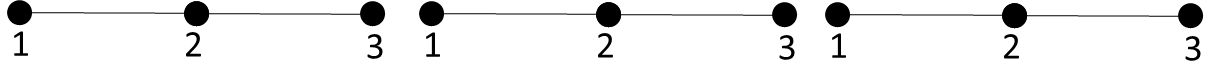}
		\caption{A graph with $9$ vertices with a highly $2$-resistant vertex $3$-multicoloring} 
		\label{fig:f2}\ \
	\end{figure}
	
	This proves the Theorem.
\end{proof}

Let us prove two auxiliary Lemmas.

\begin{lemma}\label{lm4}
	Let $G$ be a graph with at most $7$ vertices, different from cycle
	$C_{7}$, $k\in\mathbb{N}$, and $\kappa:V(G)\rightarrow\mathcal{P}(\{1,...,k\})$\ a vertex $k$-multicoloring. Then one
	of the following holds:
	
	i) The $3-HR$ condition doesn't hold for $\kappa$ and $G$;
	
	ii) $\kappa$ is not $1$-resistant vertex $k$-multicoloring of $G$.
	
\end{lemma}

\begin{proof}
	If there is a vertex $v$ of degree at least $3$ in $G$, then
	$G\backslash M(v)$ has at most three vertices. If 
	${\displaystyle\bigcup\limits_{u\in G\backslash M(v)}\kappa(u)=\{1,...,k\}}$ then (i) holds, and if
		${\displaystyle\bigcup\limits_{u\in G\backslash M(v)}}\kappa(u)\neq\{1,...,k\}$ then (ii) holds.
	
	Let us assume $\Delta(G)\leq2$. If $G$ is connected,
	then $G$ is a path. Let $v$ be a central vertex (or one of two central
	vertices) of this path. Note that each component of $G\backslash M(v)$ has at
	most two vertices. But then again, as above, (i) or (ii) holds. Therefore, let us
	assume that $G$ is not connected. As for each component with less then $4$ vertices (i) or (ii) holds, let us observe the largest
	component $H$, and assume that it has at least $4$ vertices. Note that it has
	at most $6$ vertices and a vertex $w$ of degree at least $2$. But then
	$G\backslash M(w)$ has at most $3$ vertices, so again (i) or (ii) holds.
\end{proof}

\begin{lemma}\label{lm5}
	Let $\kappa:V(C_{7})\rightarrow\mathcal{P}(\{1,...,6\})$ be a vertex
	$6$-multicoloring. Then one of the following holds:
	
	i) The $3-HR$ condition doesn't hold for $\kappa$ and $C_{7}$;
	
	ii) $\kappa$ is not $1$-resistant vertex $6$-multicoloring of $C_{7}$.
\end{lemma}

\begin{proof}
	Let us suppose (i) doesn't hold and let us denote the vertices in the cycle by
	$w_{0},...,w_{6}$ so that $w_{i}w_{i+1}\in E$ for $i\in \{0,...,5\}$ and $w_{6}w_{0}\in E$. Let $W_{i}=\{w_{i},w_{i+1},w_{i+2}\}$ where summation is
	taken modulo $7$. Note that because (i) doesn't hold, each set $W_{i}$ has to not have at least one color, that is 
	${\displaystyle\bigcup\limits_{u\in W_{i}}\kappa(u)\neq\{1,...,6\}}$, for each $i\in \{1,...,6\}$.
	Since there are $7$ sets and $6$ colors, there are two sets $W_{i}$ and $W_{j}$
	that miss the same color, but then there are at least four consecutive vertices
	that are missing the same color. Let us denote them by $x_{i},x_{i+1}%
	,x_{i+2},x_{i+3}$. Now if we observe $C_{7}\backslash M(x_{i+5})=\{x_{i},x_{i+1},x_{i+2}%
	,x_{i+3}\}$, we see that (ii) holds.
\end{proof}

Now, we can prove:

\begin{theorem}\label{tm6}It holds
	\[
	K(3,n)=\left\{
	\begin{array}
	[c]{cc}%
	\infty, & n\leq13\text{;}\\
	7, & n=14,15\text{;}\\
	4, & n\geq16\text{.}%
	\end{array}
	\right.
	\]
\end{theorem}

\begin{proof}
	First, let us prove that $K(3,n)=\infty$, for each $n\leq13$, that is that a graph with $n\leq13$ vertices cannot have a highly $3$-resistant vertex $k$-multicoloring, for any $k\in \mathbb{N}$.\\ Suppose to the
	contrary, that there exists $k\in\mathbb{N}$, a graph $G$ with $13$ vertices and a multicoloring of vertices of $G$ with $k$ colors that is highly $3$-resistant. Note that $\kappa$ is  not highly $3$-resistant in components with at most three vertices, hence we may restrict our attention to the components with at most four vertices. There are
	at most three such components. Let us distinguish three cases.
	
	1) There are three components with at least 4 vertices.
	
	Note that each of these components has at most 5 vertices, so the claim
	follows from Lemma \ref{lm4}.
	
	2) There are exactly two components with at least 4 vertices.
	
	Let $H_{1}$ be the smaller of these components (it has at most $6$ vertices). If for $H_{1}$ and $\kappa$ the $3-HR$ condition doesn't hold then $\kappa$ is not highly $3$-resistant, which is a contradiction, so from Lemma \ref{lm4} it follows that $\kappa$ is not $1$-resistant in $H_{1}$. Component $H_{2}$ has at most $9$ vertices and
	$H_{2}$ has a vertex $w$ of degree at least $2$. Note that $H_{2}%
	\backslash M(w)$ has at most $6$ vertices. Again from Lemma \ref{lm4} it
	follows that $\kappa$ in $H_{2}
	\backslash M(w)$ is not $1$-resistant. Hence, $\kappa$ in $G$ is not $3$-resistant, which is a contradiction.
	
	3) There is exactly one component $H$ with at least 4 vertices.
	
	First, suppose that $\Delta(H)\leq2$, i.e. that it is either a cycle or a
	path. Note that in both cases there are three vertices $A=\{w_{1},w_{2}%
	,w_{3}\}$ such that no component of $H\backslash M(A)$ has more than $3$
	vertices, but then $\kappa$ is obviously not highly $3$-resistant. Hence, there is a vertex $w$ of degree at least
	$3$. Graph $G\backslash M(w)$ has at most $9$ vertices, and we can follow the same conclusions as in the case 2).
	
	Hence indeed, $K(3,n)=\infty$ for $n\leq13$.
	
	Now, let us prove that $K(3,n)=7$ for $p=14,15$. First, let us prove that a graph with $15$ vertices doesn't admit a highly $3$-resistant vertex $6$-multicoloring. Suppose to the contrary, and let $G$ be a graph with $15$ vertices and $\kappa$ a highly $3$-resistant vertex $6$-multicoloring. Again, note that for components with at most three vertices the $3-HR$ condition doesn't hold, so we may restrict our attention to the components with
	at lest four vertices. There are at most three such components. Let us
	distinguish three cases.
	
	a) There are three components with at least 4 vertices.
	
	If at least one of the components is $C_{7}$ the claim follows from Lemmas
	\ref{lm4} and \ref{lm5}, and if that is not the case, the claim follows from
	Lemma \ref{lm4}.
	
	b) There are exactly two components with at least 4 vertices.
	
	The larger component (say $H_{2}$) has at most $11$ vertices. If it is a cycle
	or a path, then there are two vertices $w_{2},w_{3}$ such that no component of
	$H_{2}\backslash M(\{w_{2},w_{3}\})$ has more than $3$ vertices. Let $w_{1}$
	be the vertex with the highest degree in the second largest component $H_{1}$.
	
	If we observe vertices $w_{1},w_{2},w_{3}$ we see that no component of $G\backslash M(\{w_{1},w_{2}%
	,w_{3}\})$ has all the colors or the $3-HR$ condition wouldn't hold. But then $G$ is not $3$-resistant.
	Therefore, $H_{2}$ is neither a cycle nor a path,
	hence $\Delta(H_{2})\geq3$. Hence, there is a vertex $w_{2}$ such that
	$H_{2}\backslash M(w_{2})$ has at most $7$ vertices. As $3-HR$ condition must hold for $H_{2}\backslash M(w_{2})$, from Lemmas
	\ref{lm4} and \ref{lm5}, it follows that $\kappa$ is not $1$-resistant in $H_{2}\backslash M(w_{2})$. Let us denote by $w_{3}$ the vertex such that $H_{2}\backslash M(w_{2})\backslash M(w_{3})$ doesn't have a component with all the colors, and let $w_{1}$
	be the vertex with the highest degree in the second largest component $H_{1}$.
	Now, no component of $G\backslash M(\{w_{1},w_{2},w_{3}\})$ has all the colors
	and the contradiction is obtained.
	
	c) There is exactly one component $H_{1}$ with at least $4$ vertices.
	
	If this component is either a cycle or a path, there is a set of three
	vertices $A=\{w_{1},w_{2},w_{3}\}$ such that no component of $H_{1}\backslash
	M(A)$ has more than $2$ vertices, so $\kappa$ is obviously not highly $3$-resistant. Therefore, we may assume that there is a
	vertex $w_{1}$ in $H_{1}$ of degree at least $3$. Hence, $H_{1}\backslash
	M(w_{1})$ has at most $11$ vertices. If $H_{1}\backslash M(w_{1})$ is
	either a cycle or a path, then there are vertices $w_{2}$ and $w_{3}$ such
	that $(H_{1}\backslash M(w_{1}))\backslash M(\{w_{2},w_{3}\})$ has no
	component with more than $3$ vertices, but then again either $\kappa$ is not $3$-resistant in $G$ or the $3-HR$ condition doesn't hold. Hence, we
	may conclude that $H_{1}\backslash M(w_{1})$ has a vertex $w_{2}$ of
	degree at least $3$. Now the graph $(H_{1}\backslash M(w_{1}%
	))\backslash M(w_{2})=H_{1}\backslash M(\{w_{1},w_{2}\})$ has at most
	$7$ vertices. Since $3-HR$ condition must hold for $H_{1}\backslash M(\{w_{1},w_{2}\})$, from Lemmas \ref{lm4} and \ref{lm5} it follows that $H_{1}\backslash M(\{w_{1},w_{2}\})$ is not $1$-resistant. But then $\kappa$ is not highly $3$-resistant in $H_{1}$, and therefore also not in $G$.
	
	Hence indeed, for a graph with $15$ vertices a highly $3$-resistant vertex $6$-multicoloring doesn't exist. From Theorem \ref{tm1} it follows that this is also true for graphs with $14$ vertices.
	
	In order to prove that $K(3,n)=7$, where $n=14,15$, it is sufficient to prove
	that $K(3,14)=7$. Let us present the graph $G$ with $14$ vertices and a highly $3$-resistant vertex $7$-multicoloring of $G$. Let $G$ be a union of two cycles of length $7$ in which the
	vertices are denoted $v_{0}^{j},v_{1}^{j},...,v_{6}^{j}$, $j=1,2$, and let $\{1,...,7\}$ be the set of colors. Let vertex
	$v_{i}^{j}$ have the set of colors $\{i,i+3\}, i=1,...,7$, where all operations are modulo
	$7$. Note that no three vertices can have all the colors, because there are $7$
	colors and every vertex has just two of them. Let $A$ be the set of any $3$
	vertices in $G$. Without the loss of generality we may assume that there is at
	most one vertex from $A$ in the cycle $v_{0}^{1},v_{1}^{1},...,v_{6}^{1}$.
	Hence, there is a path of $4$ vertices $v_{l}^{1},v_{l+1}^{1}v_{l+2}%
	^{1}v_{l+3}^{1}$ in the component of $G\backslash M(A)$. Note that these $4$
	vertices have all the keys. This realization is given in Figure \ref{fig:3_14_7}. This completes the proof that $K(3,n)=7$ for
	$n=14,15$.
	
		\begin{figure}[h]
		\centering\includegraphics[scale=0.7]{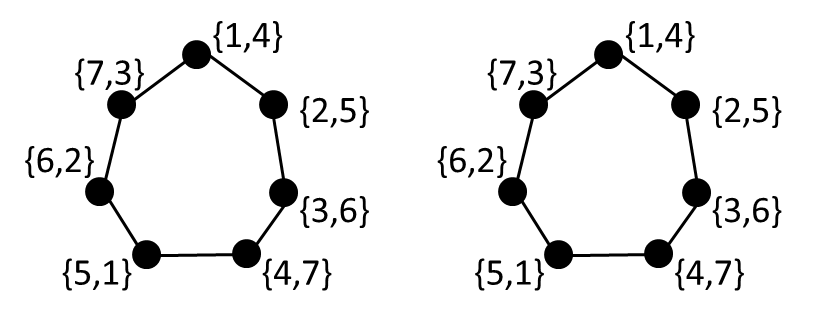}
		\caption{A graph $G$ with $14$ vertices and a highly $3$-resistant vertex $7$-multicoloring of $G$}
		\label{fig:3_14_7} \ \
	\end{figure}
	
	Finally, it holds that a graph with $n$ vertices doesn't admit a highly $3$-resistant vertex $3$-multicoloring for any $n\in\mathbb{N}$, as commented at the beginning.
	It remains to show that a graph with $16$ vertices exists that admits a highly $3$-resistant vertex $4$-multicoloring. This is shown in Figure \ref{fig:f3}.

	\begin{figure}[h]
		\centering\includegraphics[scale=0.5]{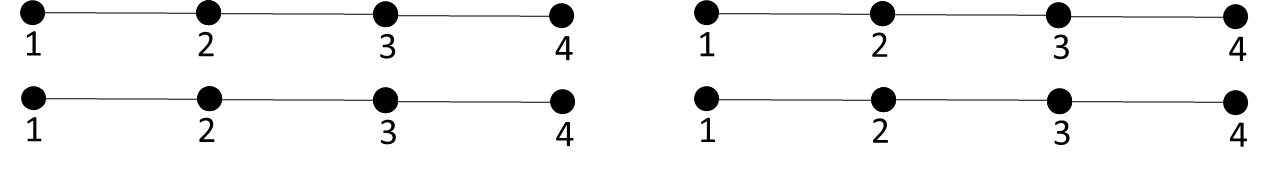}
		\caption{A graph with $16$ vertices exists that admits a highly $3$-resistant vertex $4$-multicoloring}
		\label{fig:f3} \ \
	\end{figure}

\end{proof}

Let us prove an auxiliary Lemma:

\begin{lemma}
	\label{lm7}Let $G$ be a graph with at most $8$ vertices, different from
	$C_{8}$, and let $\kappa:V(G)\rightarrow\{1,...,k\}$ be a vertex $k$-multicoloring.
	Then one of the following holds:
	
	i) The $4-HR$ condition doesn't hold for $\kappa$ and $G$;
	
	ii) $\kappa$ is not $1$-resistant vertex $k$-multicoloring of $G$.

\end{lemma}

\begin{proof}
	Suppose to the contrary, let $G$ be a graph with at most $8$ vertices and $\kappa:V(G)\rightarrow\{1,...,k\}$ be a vertex $k$-multicoloring of $G$ such that $\kappa$ is $1$-resistant and the $4-HR$ condition holds. It follows that no component with at most $4$ vertices has all the colors.
	Let us observe the largest component, $H_{1}$. If $H_{1}$ is
	a path with at most $8$ vertices, or a cycle with at most $7$ vertices, it can
	be easily seen that there is a vertex $v_{0}$ such that no component of
	$H_{1}\backslash M(v_{0})$ has more than $4$ vertices. If some component of $H_{1}\backslash M(v_{0})$ has
	all the colors, then the $4-HR$ doesn't hold, and if not, then $H_{1}$ is not $1$-resistant and therefore nor is $G$. Hence, we may assume
	that $H_{1}$ has a vertex $w$ of degree at least $3$. Then, $H_{1}\backslash
	M(w)$ has at most $4$ vertices and the claim follows as before. 
\end{proof}

\begin{theorem}
	\label{tm8}
	$K(4,n)=\infty$ for $n\leq20$.
\end{theorem}

\begin{proof}
	Suppose to the contrary, that there exists $k\in\mathbb{N}$, a graph $G$ with $20$ vertices and a multicoloring of vertices of $G$ with $k$ colors that is highly $4$-resistant. 
	Let us distinguish two cases, depending on the maximal degree of $G$.
	
	1) $\Delta(G)\leq2$.
	
	In this case, $G$ is a union of components that are paths and cycles. Note that $4-HR$
	implies that no component with at most $4$ vertices can have all the colors.
	Hence, we restrict our attention to the components with at least $5$ vertices. Let
	us distinguish four subcases.
	
	1.1) There are four components with at least $5$ vertices.
	
	Note that each of these components has exactly $5$ vertices. Let us denote
	these components by $H_{1},...,H_{4}$ and let $w_{i}\in H_{i}$, $i=1,...,4$,
	be arbitrary vertices. Note that graph $G\backslash M(\{w_{1},...,w_{4}\})$
	has no component with at least $5$ vertices, but then $\kappa$ cannot be $4$-resistant and highly $4$-resistant in $G$.
	
	1.2) There are exactly three components with at least $5$ vertices.
	
	Let us denote these components by $H_{1},H_{2}$ and $H_{3}$ in such a way that
	$v(H_{1})\leq v(H_{2})\leq v(H_{3})$. It follows $v(H_{1}
	),v(H_{2})\leq7$. As $4-HR$ condition holds for $\kappa$ and $G$, Lemma \ref{lm7} implies that $H_{1}$ and $H_{2}$ are not $1$-resistant, that is, there are
	vertices $w_{1}\in V(H_{1})$ and $w_{2}\in V(H_{2})$ such that $H_{i}%
	\backslash M(w_{i})$, $i=1,2$, does not have all the colors. Further, note that
	$v(H_{3})\leq10$. Hence, there are two vertices $w_{3}$ and $w_{4}$ such
	that each component of $H_{3}\backslash M(\{w_{3},w_{4}\})$ has at most two
	vertices. Therefore, no component of $G\backslash M(\{w_{1},...,w_{4}\})$ can
	have all the colors.
	
	1.3) There are exactly two components with at least $5$ vertices.
	
	Let us denote these components by $H_{1}$ and $H_{2}$ in such a way that
	$v(H_{1})\leq v(H_{2})$. If $v(H_{1}),v(H_{2})\leq14$ then there are
	four vertices $w_{1},w_{2}\in V(H_{1})$ and $w_{3},w_{4}\in V(H_{2})$ such
	that no component of $H_{1}\backslash M(\{w_{1},w_{2}\})$ has more than $4$
	vertices and that no component of $H_{2}\backslash M(\{w_{3,}w_{4}\})$ has
	more than 4 vertices, so the claim follows as before. If $v(H_{1})=5$ and $v(H_{2})=15$, the
	proof goes analogously, except that we observe $w_{1}\in V(H_{1})$ and
	$w_{2},w_{3},w_{4}\in V(H_{2})$.
	
	1.4) There is exactly one component with at least $5$ vertices.
	
	It can be easily seen that $4$ vertices, $w_{1},...,w_{4}$, can be found, such
	that no component of $G\backslash M(\{w_{1},...,w_{4}\})$ has more than $4$
	vertices which leads to a contradiction.
	
	This completes the proof of case 1).
	
	2) $\Delta(G)\geq3$.
	
	Let $v_{1}$ be a vertex in $G$ such that $d(v_{1})=\Delta(G)$. Note that
	$G_{2}=G\backslash M(v_{1})$ has at most $16$ vertices. Let us distinguish
	three subcases.
	
	2.1) $\Delta(G_{2})\leq2$ and $G_{2}$ is not isomorphic to $C_{8}\cup C_{8}$.
	
	As before, we observe only components with at least $5$ vertices. We
	distinguish four subcases.
	
	2.1.1) If there are three components $H_{1},H_{2}$ and $H_{3}$, then each of
	them has at most $6$ vertices, and from Lemma \ref{lm7} it follows that neither of them is $1$-resistant, that is,
	there are vertices $w_{i}\in V(H_{i})$ such that no component of
	$H_{i}\backslash M(w_{i})$ has all the colors which leads to contradiction with $G$ being highly $4$-resistant.
	
	2.1.2) If there are two components $H_{1}$ and $H_{2}$, such that
	$v(H_{1})\leq v(H_{2})$ and $v(H_{1})\leq7$, then again Lemma
	\ref{lm7} implies that there is a vertex $w_{1}\in V(H_{1})$ such that no
	component $H_{1}\backslash M(w_{1})$ of has all the colors. Also, there are two
	vertices $w_{2}$ and $w_{3}$ such that no component of $H_{2}\backslash
	M(\{w_{2},w_{3}\})$ has more than
	\[
	\left\lceil \dfrac{11-2\cdot3}{2}\right\rceil =3
	\]
	vertices. This again leads to a contradiction.
	
	2.1.3) If there are two components $H_{1}$ and $H_{2}$ such that
	$v(H_{1})=v(H_{2})$, then at least one of them (say $H_{1}$) is not a
	cycle, but a path. Hence, there is a vertex $w_{1}\in V(H_{1})$ such that no
	component of $H_{1}\backslash M(w_{1})$ has more than
	\[
	\left\lceil \dfrac{8-3}{2}\right\rceil =3
	\]
	vertices, and there are vertices $w_{2},w_{3}\in V(H_{2})$ such that no
	component of $H_{2}\backslash M(\{w_{2},w_{3}\})$ has more than
	\[
	\left\lceil \dfrac{8-2\cdot3}{2}\right\rceil =1
	\]
	vertex. This leads to a contradiction.
	
	2.1.4) If there is only one component with more than $4$ vertices, then there
	are three vertices $w_{1},w_{2},w_{3}$ such that each component of
	$G_{2}\backslash M(\{w_{1},w_{2},w_{3}\})$ has at most
	\[
	\left\lceil \dfrac{16-3\cdot3}{3}\right\rceil =3
	\]
	vertices, which leads to a contradiction.
	
	2.2) $\Delta(G_{2})\geq3$.
	
	Let $v_{2}$ be a vertex in $G_{2}$ such that $d(v_{2})=\Delta(G_{2})$.
	
	Note that $G_{3}=G_{2}\backslash M(v_{2})$ has at most $12$ vertices. Let us
	distinguish two subcases.
	
	2.2.1) $\Delta(G_{3})\leq2$.
	
	Note that $G_{3}$ has at most two components with at least $5$ vertices. If
	$G_{3}$ has only one component with at least $5$ vertices, then there are two
	vertices $w_{3}$ and $w_{4}$ such that $G_{3}\backslash M(\{w_{3},w_{4}\})$
	has no more than
	\[
	\left\lceil \dfrac{12-3\cdot2}{2}\right\rceil =3
	\]
	vertices, which leads to a contradiction. If there are two components $H_{1}$
	and $H_{2}$ with at least $5$ vertices, then each of them has at most $7$
	vertices so from Lemma \ref{lm7} it follows that there is a
	vertex $w_{3}$ such that no component of $H_{1}\backslash M(w_{3})$ has all
	the colors, and there is a vertex $w_{4}$ such that no component of
	$H_{2}\backslash M(w_{4})$ has all the colors, which also leads to a contradiction.
	
	2.2.2) There is a vertex $w_{3}$ of degree at least $3$ such that
	$G_{3}\backslash M(w_{3})$ is not isomorphic to $C_{8}$.
	
	Note that the graph $G_{4}=G_{3}\backslash M(w_{3})$ has at most $8$ vertices.
	From Lemma \ref{lm7} it follows that there is a vertex $w_{4}$ such
	that no component of $G_{4}\backslash M(w_{4})$ has all the colors, which leads
	to contradiction.
	
	2.2.3) For each vertex $w_{3}$ of degree at least $3$ it holds that
	$G_{3}\backslash M(w_{3})$ is isomorphic to $C_{8}$.
	
	Let us denote the vertices of $G_{3}\backslash M(w_{3})$ by $u_{1},u_{2},...,u_{8}$, so that $u_{i}u_{i+1}\in E$ for $i\in \{1,...,8\}$ and $u_{8}u_{1}\in E$ and let us denote neighbors of $w_{3}$
	by $x_{1},x_{2},x_{3}$. If there is no edge between some $u_{i}$ and some
	$x_{i}$ then $G\backslash M(\{w_{1},w_{2},u_{1},u_{4}\})$ doesn't have a
	component with more then 4 vertices and that would lead to a contradiction.
	Hence, without the loss of generality, we may assume that $x_{1}u_{1}\in
	E(G)$. Hence, $u_{1}$ is a vertex of degree at least $3$. Therefore,
	$G_{3}\backslash M(u_{1})$ is isomorphic to $C_{8}$. Without the loss of generality, we may
	assume that $u_{7}x_{2},$ $u_{3}x_{3}\in E(G_{3})$ (other option, that
	$u_{7}x_{3}$ and $u_{3}x_{2}$ are edges in $G_{3}$ is analogous). Hence, the
	edges of $G_{3}$ are
	\[
	\{u_{i}u_{i+1}:i=1,...,7\}\cup\{u_{8}u_{1}\}\cup\{w_{3}x_{i}:i=1,2,3\}\cup
	\{u_{1}x_{1},u_{7}x_{2},u_{3}x_{3}\}\text{,}%
	\]
	but then $u_{7}$ has the degree $3$ and $G_{3}\backslash M(u_{7})$ is not
	isomorphic to $C_{8}$, which is a contradiction.
	
	2.3) $G_{2}$ is isomorphic to $C_{8}\cup C_{8}$.
	
	Let us denote by $u_{1}^{j},u_{2}^{j},...,u_{8}^{j}$, $j=1,2$ vertices of
	these two cycles in order of their appearance and let us denote any $3$
	neighbors of $w_{1}$ in $G$ by $x_{1},x_{2},x_{3}$. If there is no edge
	between some $x_{p}$ and some $u_{i}^{j}$, $p=1,2,3$, $j=1,2$, $i=1,...,8$,
	then $G\backslash M(\{u_{1}^{1},u_{4}^{1},u_{1}^{2},u_{4}^{2}\})$ has no
	component with at least $5$ vertices which leads to a contradiction. Hence, let
	us assume, without the loss of generality, that $x_{1}u_{1}^{1}\in E(G)$. Then
	$u_{1}^{1}$ has the degree $3$. If $G\backslash M(u_{1}^{1})$ is not
	isomorphic to $C_{8}+C_{8}$, then one of the previous cases holds. If
	$G\backslash M(u_{1}^{1})$ is isomorphic to $C_{8}\cup C_{8}$, then we may assume
	(similarly as above), without the loss of generality, that $u_{7}^{1}x_{2}$,
	$u_{3}^{1}x_{3}\in E(G)$, but then it can be shown that $G\backslash
	M(u_{7}^{1})$ is not isomorphic to $C_{8}\cup C_{8}$. For $G\backslash M(u_{7}%
	^{1})$ one of the previous cases holds.
	
	This completes the proof of the Theorem.
\end{proof}

In order to determine $K(4,21)$ we need the following Lemmas.

\begin{lemma}
	\label{lm9}Let $\kappa:V(C_{8})\rightarrow\{1,...,9\}$ be a vertex $9$-multicoloring. Then one of the following holds:
	
	i) The $4-HR$ condition doesn't hold for $\kappa$ and $C_{8}$;
	
	ii) $\kappa$ is not $1$-resistant vertex $9$-multicoloring of $C_{8}$.

\end{lemma}

\begin{proof}
	Suppose to the contrary, that neither (i) nor (ii) holds. It means that  $\kappa:V(C_{8})\rightarrow\{1,...,9\}$ is a $1$-resistant multicoloring for which the $4-HR$ condition holds. 
	Let us denote vertices of this cycle by
	$u_{1},...,u_{8}$ so that $u_{i}u_{i+1}\in E$ for $i\in \{1,...,7\}$ and $u_{8}u_{1}\in E$. Note that each four consecutive vertices $u_{p}
	,u_{p+1},u_{p+2},u_{p+3}$ don't have all the colors, because of the $4-HR$. Let us denote by $k_{p}$
	one of the colors that they don't have. Note that $u_{p-1}$ must have the color
	$k_{p}$, because otherwise no vertex in $C_{8}\backslash M(u_{p-3})$ would
	have color $k_{p}$, and $\kappa$ wouldn't be $1$-resistant. 
	Analogously note that $u_{p+4}$ must have the color $k_{p}$.
	
	Further, let us prove that $p\neq q$ implies $k_{p}\neq k_{q}$. Suppose to the
	contrary. Then it holds that no vertex in $\{u_{q},...,u_{q+3}\}$ has the color
	$k_{p}=k_{q}$ and both vertices in $\{u_{p-1},u_{p+4}\}$ have this color, but
	$\{u_{q},...,u_{q+3}\}\cap\{u_{p-1},u_{p+3}\}\neq\emptyset$ so that is a
	contradiction with the number of vertices. Hence, without the loss of
	generality, we may assume that $k_{p}=p$. Therefore, $\kappa(u_{1})\cup
	\kappa(u_{3})\cup\kappa(u_{5})\cup\kappa(u_{7})\supseteq\{1,..,8\}$ and
	$\kappa(u_{2})\cup\kappa(u_{4})\cup\kappa(u_{6})\cup\kappa(u_{8}%
	)\supseteq\{1,..,8\}$. Hence, no matter which vertex has the color $9$, there
	will be four vertices that have all the colors which is contradiction with $4-HR$.
\end{proof}

\begin{lemma}
	\label{lm10}Let $G$ be a graph with at most $8$ vertices and let
	$\kappa:V(G)\rightarrow\{1,...,9\}$ be a vertex $9$-multicoloring. One of the following holds:
	
	i) The $4-HR$ condition doesn't hold for $\kappa$ and $G$;
	
	ii) $\kappa$ is not $1$-resistant vertex $9$-multicoloring of $G$.	
\end{lemma}

\begin{proof}
	The claim follows from Lemmas \ref{lm7} and \ref{lm9}.
\end{proof}

\begin{lemma}
	\label{lm11}Let $G$ be a graph with at most $12$ vertices and let
	$\kappa:V(G)\rightarrow\{1,...,9\}$ be a vertex $9$-multicoloring. One of the following holds:
	
	i) The $4-HR$ condition doesn't hold for $\kappa$ and $G$;
	
	ii) $\kappa$ is not $2$-resistant vertex $9$-multicoloring of $G$.
\end{lemma}

\begin{proof}
	Suppose to the contrary. If there is a vertex $x_{1}$ of degree at least $3$,
	then $G_{2}=G\backslash M(x_{1})$ has at most $8$ vertices and from
	Lemma \ref{lm10} it follows that there is a vertex $x_{2}$ such that no
	component of $G_{2}\backslash M(x_{2})=G\backslash M(\{x_{1},x_{2}\})$ has all
	the colors, and $\kappa$ is not $2$-resistant, which is a contradiction. Hence, $\Delta(G)\leq2$. Note that no
	component with at most $4$ vertices can have all the colors, hence we observe
	only components with $5$ or more vertices. If there are two components $H_{1}$
	and $H_{2}$ with at least $5$ vertices, then Lemma \ref{lm7} implies that
	there are vertices $x_{i}\in H_{i}$ such that no component of $H_{i}\backslash
	M(x_{i})$, $i=1,2$, has all the colors, which leads to a contradiction. If there is
	only one component $H$ with at least $5$ vertices, then there are vertices
	$x_{1},x_{2}\in H$ such that no component of $H\backslash M(\{x_{1},x_{2}\})$
	has more than
	\[
	\left\lceil \dfrac{12-2\cdot3}{2}\right\rceil =3
	\]
	vertices, which also leads to a contradiction.
\end{proof}

\begin{lemma}
	\label{lm12}Let $G$ be a graph with at most $16$ vertices and let
	$\kappa:V(G)\rightarrow\{1,...,9\}$ be a vertex $9$-multicoloring. One of the following holds:
	
	i) The $4-HR$ condition doesn't hold for $\kappa$ and $G$;
	
	ii) $\kappa$ is not $3$-resistant vertex $9$-multicoloring of $G$.	
\end{lemma}

\begin{proof}
	Suppose to the contrary. If there is a vertex $x_{1}$ of degree at least $3$,
	then $G_{2}=G\backslash M(x_{1})$ has at most $12$ vertices and from
	Lemma \ref{lm10}, it follows that there are vertices $x_{2}$ and $x_{3}$ such
	that no component of $G_{2}\backslash M(\{x_{2},x_{3}\})=G\backslash
	M(\{x_{1},x_{2},x_{3}\})$ has all the colors, which is a contradiction. Hence,
	$\Delta(G)\leq2$. Note that no component with at most $4$ vertices can have
	all the colors, hence we observe only components with $5$ or more vertices. If
	there are three components $H_{1}$, $H_{2}$ and $H_{3}$ with at least $5$
	vertices each, then Lemma \ref{lm10} implies that there are vertices $x_{i}\in
	H_{i}$ such that no component of $H_{i}\backslash M(x_{i})$, $i=1,2,3$, has
	all the colors, which leads to a contradiction. Suppose that there are exactly two
	components $H_{1}$ and $H_{2}$, with at least $5$ vertices each. Without the
	loss of generality, we may assume that $v(H_{1})\leq v(H_{2})$. Note that
	$H_{1}$ has at most $8$ vertices, and hence Lemma \ref{lm10} implies that
	there is a vertex $x_{1}\in V(H_{1})$ such that no component of $H_{1}%
	\backslash M(x_{1})$ has all the colors. $H_{2}$ has at most $11$ vertices, and
	Lemma \ref{lm11} implies that there are vertices $x_{2},x_{3}\in V(H_{2})$
	such that no component of $H_{2}\backslash M(\{x_{2},x_{3}\})$ has all the
	colors. If there is only one component, $H$, with at least $5$ vertices, then
	there are vertices $x_{1},x_{2},x_{3}\in V(H)$ such that no component of
	$H\backslash(\{x_{1},x_{2},x_{3}\})$ has more than
	\[
	\left\lceil \dfrac{16-3\cdot3}{3}\right\rceil =3
	\]
	vertices, which also leads to a contradiction.
\end{proof}

\begin{theorem}
	\label{tm13}$K(4,21)=10$.
\end{theorem}

\begin{proof}
	First, let us show that there exists a graph $G$ with $21$ vertices and a highly $4$-resistant vertex $10$-multicoloring $\kappa$ of $G$. Let $G$ be a graph that has three
	components; two cycles $H_{1}$ and $H_{2}$ with $8$ vertices and a path $H_{3}$
	with $5$ vertices. Let vertices of cycle $H_{i}$ be denoted by $v_{1}%
	^{i},...,v_{8}^{i}$, $i=1,2$, and let vertices of $H_{3}$ be denoted by
	$v_{1}^{3},...,v_{5}^{3}$, with the adjacency defined as in previous proofs. For $i=1,2,3$ let function $\kappa$ be defined by%
	\[
	\kappa(v_{j}^{i})=\left\{
	\begin{array}
	[c]{cc}%
	\{j,j+3,9\}, & j\text{ odd;}\\
	\{j,j+3,10\} & j\text{ even,}%
	\end{array}
	\right.
	\]

	where summation is given modulo 8. This way every vertex has two colors in
	$\{1,...,8\}$. This multicoloring is represented in Figure \ref{fig:4_21_10}.
	
		\begin{figure}[h]
		\centering\includegraphics[scale=0.6]{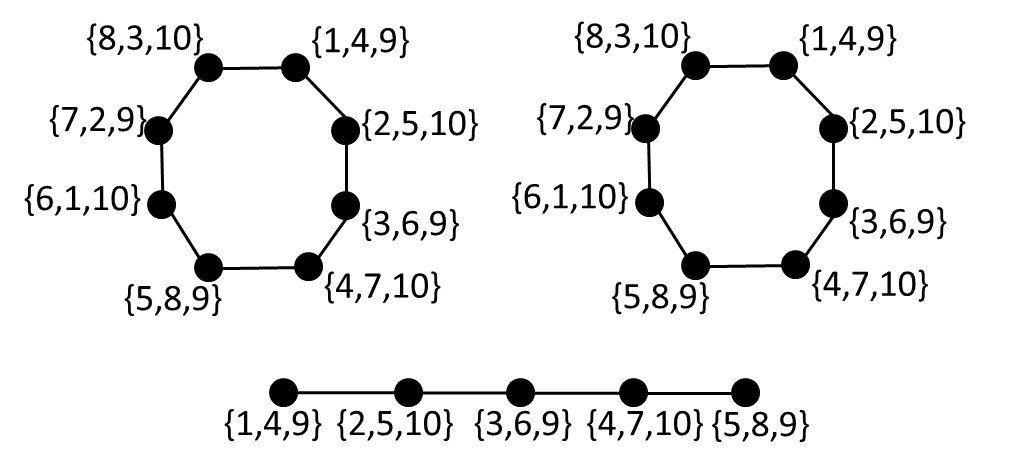}
		\caption{A graph $G$ with $21$ vertices and a highly $4$-resistant vertex $10$-multicoloring}
		\label{fig:4_21_10} \ \
	\end{figure}

	Let us prove that $4-HR$ holds for this $G$ and $\kappa$. Suppose to the contrary, that
	there is a set $A\subseteq V(G)$ of $4$ vertices that has all $10$ colors. Without the
	loss of generality we may assume that all the vertices in $A$ are vertices in
	$H_{1}$, since $\kappa$ does not depend on the component. Each vertex has $2$
	colors in $\{1,...,8\}$, hence all these colors must be different. Let us
	distinguish three cases.
	
	1) There are two vertices in $A$ on distance $3$.
	
	They both have the same color in $\{1,...,8\}$ which is a contradiction.
	
	2) There are no two vertices on distance $3$ in $A$, but there are two
	vertices on the distance $1$.
	
	Without the loss of generality, we may assume that these vertices are
	$v_{1}^{1}$ and $v_{2}^{1}$. Hence, $v_{4}^{1},v_{5}^{1},v_{6}^{1},v_{7}%
	^{1}\notin A$. Therefore, $v_{3}^{1},v_{8}^{1}\in A$, but this is a
	contradiction, because these two vertices are on the distance $3$.
	
	3) All vertices in $A$ are on even distances.
	
	In this case, either none of the vertices has the color $9$ or none of them has the color
	$10$. In both cases contradiction is obtained.
	
	So $4-HR$ holds for defined $\kappa$ and $G$. In order to prove $\kappa$ is highly $4$-resistant it remains to prove it is $4$-resistant.
	
	Suppose to the contrary, that there is a set
	$A$ of $4$-vertices such that $G\backslash M(A)$ has no component with all the colors. 
	It is easy to see that at least one of the vertices in $A$ is in $H_{3}%
	$. Hence, there are at most $3$ vertices in $A$ in $H_{1} \cup H_{2}$.
	Without the loss of generality, we may assume that there is at most one vertex
	in $A$ in $H_{1}$ and let us denote it by $w$. It can be easily seen that
	$H_{1}\backslash M(w)$ has all the colors which is a contradiction. Hence indeed
	a graph $G$ with $21$ vertices and a highly $4$-resistant vertex $10$-multicoloring exists.\\
	
	It remains to prove that a highly $4$-resistant vertex $9$-multicoloring doesn't exist for any grapf with $21$ vertices. Suppose to the contrary, that
	there is a graph $G$ with $21$ vertices and a highly $4$-resistant multicoloring $\kappa$ of vertices of $G$ with $9$ colors. Let us distinguish $5$ cases.
	
	a) $G$ has at least one component with at most $4$ vertices.
	
	This component, $H_{1}$, does not have all the colors, hence we observe
	$G\backslash H_{1}$. Now contradiction follows from Theorem \ref{tm8}.
	
	b) $G$ is not connected, but all the components have at least $5$ vertices.
	
	Note that there are at most $4$ components. If there are $4$ components
	$H_{i}$, $i=1,...,4$ then none of these components has more than $6$ vertices
	and hence from Lemma \ref{lm10} and $4-HR$ it follows that there is a vertex
	$w_{i}\in H_{i}$ for each $i=1,...,4$, such that no component of
	$H_{i}\backslash M(w_{i})$ has all the colors, which leads to a contradiction.
	
	If there are three components, $H_{1},H_{2}$ and $H_{3}$, then at least two of
	them (say $H_{1}$ and $H_{2}$) have at most $7$ vertices each and the largest
	of them (say $H_{3}$) has at most $11$ vertices. From $4-HR$ and Lemma \ref{lm10}
	it follows that there are vertices $w_{i}\in V(H_{i})$ such that no component of
	$H_{i}\backslash M(w_{i})$, $i=1,2$, has all the colors; and from Lemma
	\ref{lm11} and $4-HR$ it follows that there are vertices $w_{3},w_{4}\in V(H_{3})$
	such that no component of $H_{3}\backslash M(\{w_{3},w_{4}\})$ has all the colors.
	
	If there are two components, the contradiction is obtained analogously.
	
	c) $G$ is connected and $\Delta(G)\leq2$.
	
	In this case $G$ is either a cycle or a path. In each case there are four
	vertices $w_{1},...,w_{4}$ such that no component of $G\backslash
	M(\{w_{1},...,w_{4}\})$ has more than
	\[
	\left\lceil \dfrac{21-3\cdot4}{4}\right\rceil =3
	\]
	vertices and this leads to a contradiction.
	
	d) $G$ is connected and $\Delta(G)\geq4$.
	
	Let $w_{1}$ be a vertex of degree at least $4$. Then, graph $G\backslash
	M(w_{1})$ has at most $16$ vertices and from $4-HR$ and Lemma \ref{lm12} it
	follows that there are vertices $w_{2},w_{3},w_{4}\in V(G\backslash M(w_{1}))$
	such that no component of $G\backslash M(\{w_{1},...,w_{4}\})$ has all the
	colors which is a contradiction.
	
	e) $G$ is connected and $\Delta(G)=3$.
	
	Let us distinguish 4 subcases.
	
	e.1) There is a vertex $w_{1}$ such that $d_{G}(w_{1})=3$, $G\backslash M(w_{1})$ is
	disconnected and it has a component with at most $4$ vertices.
	
	Let us denote by $H_{1}$ the component with at most $4$ vertices. Obviously $H_{1}$
	doesn't have all the colors. Note that $G\backslash(M(w_{1})\cup V(H_{1}))$ has
	at most $16$ vertices. Now contradiction follows from Lemma \ref{lm12}.
	
	e.2) There is a vertex $w_{1}$ such that $d_{G}(w_{1})=3$,
	$G\backslash M(w_{1})$ is disconnected and each component has at least $5$ vertices.
	
	Since $G\backslash M(w_{1})$ has at most $17$ vertices, there are at most
	three components. If there are $3$ components $H_{1},H_{2}$ and $H_{3}$, then
	no component has more than $7$ vertices and therefore from Lemma
	\ref{lm10} and $4-HR$ it follows that there are vertices $w_{i+1}\in V(H_{i})$ such that
	no component of $H_{i}\backslash M(w_{i+1})$ has all the colors, which leads to a contradiction.
	
	If there are two components, $H_{1}$ and $H_{2}$, then, the smaller component
	(say $H_{1}$) has at most $8$ vertices and the larger component has at most
	$12$ vertices. Now contradiciton follows from Lemma \ref{lm10} and Lemma \ref{lm11}.
	
	e.3) There is a vertex $w_{1}$ such that $d_{G}(w_{1})=3$, graph
	$G\backslash M(w_{1})$ is connected and $\Delta(G\backslash M(w_{1}))=2$.
	
	Note that $G\backslash M(w_{1})$ is either a cycle or a path. In each case
	there are vertices $w_{2},w_{3},w_{4}\in V(G\backslash M(w_{1}))$ such that no
	component of $[G\backslash M(w_{1})]\backslash M(\{w_{2},w_{3},w_{4}\})$ has
	more than
	\[
	\left\lceil \frac{17-3\cdot3}{3}\right\rceil =2
	\]
	vertices, and this leads to a contradiction.
	
	e.4) For each vertex $w_{1}$ such that $d_{G}(w_{1})=3$, graph
	$G\backslash M(w_{1})$ is connected and $\Delta(G\backslash M(w_{1}))=3$.
	
	Let us denote $G_{2}=G\backslash M(w_{1})$ and let us note that it has $17$
	vertices. We distinguish four subcases.
	
	e.4.1) There is a vertex $w_{2}$ such that $d_{G_{2}}(w_{2})=3$,
	$G_{2}\backslash M(w_{2})$ is disconnected and it has a component with at most
	$4$ vertices.
	
	Let $H_{1}$ be a component that has at most $4$ vertices. It does not have all
	the colors. Graph $G_{2}\backslash\left(  M(w_{2})\cup H_{1}\right)  $ has at
	most $12$ vertices. Hence, from Lemma \ref{lm11} there are vertices $w_{3},w_{4}\in V(G_{2}%
	\backslash\left(  M(w_{2})\cup H_{1}\right)  )$ such that no component of
	$[G_{2}\backslash\left(  M(w_{2})\cup H_{1}\right)  ]\backslash M(\{w_{3}%
	,w_{4}\})$ has all the colors, which leads to a contradiction.
	
	e.4.2) There is a vertex $w_{2}$ such that $d_{G_{2}}(w_{2})=3$,
	$G_{2}\backslash M(w_{2})$ is disconnected and each of its components has at
	least $5$ vertices.
	
	Note that $G_{2}\backslash M(w_{2})$ has $13$ vertices, hence it has $2$
	components and each component has at most $8$ vertices. $4-HR$ and Lemma
	\ref{lm10} imply that there are vertices $w_{i+2}\in V(H_{i})$, $i=1,2$, such
	that no component of $H_{i}\backslash M(w_{i+2})$ has all the colors. This leads
	to a contradiction.
	
	e.4.3) There is a vertex $w_{2}$ such that $d_{G_{2}}(w_{2})=3$,
	$G_{2}\backslash M(w_{2})$ is connected and $\Delta(G_{2}\backslash
	M(w_{2}))=2$.
	
	In this case, $G_{2}\backslash M(w_{2})$ is a cycle or a path. In each case
	there are vertices $w_{3},w_{4}$ such that $[G_{2}\backslash M(w_{2}%
	)]\backslash M(\{w_{3},w_{4}\})$ has at most
	\[
	\left\lceil \frac{13-2\cdot3}{2}\right\rceil =4
	\]
	vertices, which leads to a contradiction.
	
	e.4.4) For each vertex $w_{2}$ such that $d_{G_{2}}(w_{2})=3$, graph
	$G_{2}\backslash M(w_{2})$ is connected and $\Delta(G_{2}\backslash
	M(w_{2}))=3$.
	
	Let us denote $G_{3}=G_{2}\backslash M(w_{2})$ and let us note that $G_{3}$
	has $13$ vertices. We distinguish two subcases.
	
	e.4.4.1) There is a vertex $w_{3}$ such that $d_{G_{3}}(w_{3})=3$ and
	$G_{3}\backslash M(w_{3})$ is not connected.
	
	Note that $G_{3}\backslash M(w_{3})$ has $9$ vertices. Hence, each component,
	except possibly the largest one, has at most $4$ vertices and therefore it cannot have all the colors. The largest component, $H$, has at most $8$ vertices.
	Hence, the contradiction follows from Lemma \ref{lm7}. 
	
	e.4.4.2) For each vertex $w_{3}$ such that $d_{G_{3}}(w_{3})=3$, graph
	$G_{3}\backslash M(w_{3})$ is connected.
	
	Let us denote $G_{4}=G_{3}\backslash M(w_{3})$ and let us note that $G_{4}$
	has $9$ vertices. We distinguish $5$ subcases.
	
	e.4.4.2.1) $\delta(G)=1$.
	
	Let $v_{1}$ be a vertex of degree $1$ and let $v_{2}$ be its only neighbor.
	Note that $d_{G}(v_{2})\geq2$. If $d_{G}(v_{2})=2$ let $v_{1},v_{2},...,v_{k}$
	be a path such that $d(v_{1}),...,d(v_{k-1})<3$ and $d(v_{k})=3$. Let us
	denote $w_{1}=v_{k}$. Then $G\backslash M(w_{1})$ is disconnected which is a
	contradiction with the assumptions of this subcase. Hence, let us assume that
	$d_{G}(v_{2})=3$ and let us denote $N_{G}(v_{2})=\{v_{1},p_{1},q_{1}\}$. If
	either one of vertices $p_{1}$ or $q_{1}$ has the degree $3$ (say $p_{1}$),
	then $G\backslash M(p_{1})$ is not connected which is in contradiction with
	the assumptions of this case. Hence, let $p_{1},...,p_{k}$ ($p_{2}\neq v_{2}$)
	be a path such that $d(p_{1}),...,d(p_{k-1})<3$ and $d(p_{k})=3$. Let us
	denote $w_{1}=p_{k}$. If $v_{2}\in N_{G}(w_{1})$ or $q_{1}\in N_{G}(w_{1})$,
	graph $G_{2}=G\backslash M(w_{1})$ is disconnected which is a contradiction.
	Note that $d_{G_{2}}(v_{2})=2$. Let $v_{1},v_{2},q_{1},q_{2},...,q_{l}$ be a
	path such that $d_{G_{2}}(q_{2}),...,d_{G_{2}}(q_{l-1})<3$ and $d_{G_{2}%
	}(q_{l})=3$. We denote $w_{2}=q_{l}$. Graph $G_{2}\backslash M(w_{2})$ is
	disconnected which is in contradiction with the assumptions of this case.
	
	e.4.4.2.2) $\delta(G)\geq2$ and there are two adjacent vertices of
	degree $2$.
	
	Let us assume that there are two adjacent vertices $v_{1}$ and $v_{2}$ of
	degree $2$. Let $v_{1},...,v_{k}$ be a path such that $d_{G}(v_{1}%
	),...,d_{G}(v_{k-1})<3$ and $d_{G}(v_{k})=3$. Let us distinguish $3$ subcases.
	
	e.4.4.2.2.1) $v_{1}v_{k}\notin E(G)$.
	
	Let us denote $w_{1}=v_{k}$ and $G_{2}=G\backslash M(w_{1})$. Note that
	$d_{G_{2}}(v_{k-2})=1$. Now, we can proceed analogously as in the subcase
	e.4.4.2.1 just observing $G_{2}$ instead of $G_{1}$.
	
	e.4.4.2.2.2) $v_{1}v_{k}\in E(G)$.
	
	If $k\neq3$ let us choose $w_{1}=v_{k}$. Note that $G\backslash M(w_{1})$ is
	disconnected which is a contradiction with the assumptions of this case.
	Hence, $k=3$. Let us denote $N_{G}(v_{3})=\{v_{1},v_{2},p_{1}\}$. If
	$d_{G}(p_{1})=3$, then the choice $w_{1}=p_{1}$ implies that $G\backslash
	M(w_{1})$ is disconnected, which is a contradiction. Suppose otherwise and let
	$p_{1},p_{2},...,p_{l}$ be a walk such that $d_{G}(p_{1}),...,d_{G}%
	(p_{l-1})<3$ and $d_{G}(p_{l})=3$. Then, the choice $w_{1}=p_{l}$ implies that
	$G\backslash M(w_{1})$ is disconnected, which is a contradiction.
	
	e.4.4.2.3) $\delta(G)\geq2$ and there are two adjacent vertices of
	degree $3$.
	
	Let us distinguish two subcases.
	
	e.4.4.2.3.1) $\delta(G)\geq2$ and there is a vertex $v_{2}$ of degree
	$3$ that is adjacent to at least one vertex of degree $3$ and at least one
	vertex of degree $2$.
	
	Let us denote the adjacent vertex of degree $2$ by $v_{1}$ and vertex of
	degree $3$ by $v_{3}$. Let us distinguish two subcases:
	
	e.4.4.2.3.1.1) $v_{1}v_{3}\notin E(G)$.
	
	Let us choose $w_{1}=v_{3}$ and denote $G_{2}=G\backslash M(w_{1})$. Note that
	$d_{G_{2}}(v_{1})=1$. Now we can proceed analogously as in the subcase
	e.4.4.2.1 just observing $G_{2}$ instead of $G_{1}$.
	
	e.4.4.2.3.1.1.2) $v_{1}v_{3}\in E(G)$.
	
	Let $N(v_{2})=\{v_{1},v_{3},p_{1}\}$. If $d_{G}(p_{1})=3$, then we choose
	$w_{1}=p_{1}$ and observe $G_{2}=G\backslash M(p_{1})$. If $p_{1}v_{3}\in
	E(G)$, then $G_{2}$ is disconnected which is a contradiction. In the opposite
	case $d_{G_{2}}(v_{1})=1$ and we can proceed analogously as in the subcase
	e.4.4.2.1 just observing $G_{2}$ instead of $G_{1}$.
	
	e.4.4.2.3.2) $\delta(G)\geq2$ and there is no vertex $v_{2}$ of degree
	$3$ that is adjacent to at least one vertex of degree $3$ and at least one
	vertex of degree $2$.
	
	Let us partition vertices of $G$ in three classes: class $A$ of vertices of
	degree $3$ that are adjacent to at least one vertex of degree $2$, class $B$
	of vertices of degree $3$ that are not adjacent to any vertex of degree $2$,
	and class $C$ of vertices of degree $2$. Note that no two vertices in class
	$A$ are adjacent and that no vertex in $A$ is adjacent to any vertex in $B$.
	Hence, we may assume that only adjacent vertices of degree $3$ are in $B$.
	Therefore, $B$ is non-empty and none of its vertices are adjacent to any
	vertex in $A$ or $C$. Since $G$ is connected, it must hold that $A$ and $C$
	are empty, but then the graph has $21$ vertex of degree $3$, which is in
	contradiction with handshaking Lemma which says that the number of vertices of
	odd degree in a graph is even.
	
	e.4.4.2.4) $\delta(G)\geq2$, $G$ is bipartite graph with vertices of
	degree $2$ in one partition and vertices of degree $3$ in the other partition.
	
	Let $n_{i}$, $i=2,3$, be the number of vertices of degree $i$. It must hold
	$2n_{2}=3n_{3}$ and $n_{2}+n_{3}=21$. Solving this, we get $n_{2}=63/5$ which
	is obviously a contradiction.
	
	All the cases are analyzed and the Theorem is proved.
\end{proof}

\section*{Conclusions}
We proposed a type of vertex multicoloring that could be used to model a securing of a secret situation. The multicoloring is named highly $a$-resistant vertex $k$-multicoloring, for $a,k\in\mathbb{N}$, and for given $a=1,2,3,4$, we have analyzed what is the minimal number of vertices a graph must have in order to admit a highly $a$-resistant vertex $k$-multicoloring and what is the minimal number of colors for which the coloring is achieved.
Further work includes determining minimal numbers of vertices and colors for larger number $a$ or changing the conditions of this multicoloring, motivated by different situations, resulting in new multicolorings and exploring their properties.

\section{Acknowledgements}
Partial support of the Croatian Ministry of Science and Education is gratefully acknowledged. VZ a acknowledge the Israel-Italian collaborative project NECSTas well as support by the H2020 CSA Twinning project No. 692194.

\end{document}